\documentclass[12pt]{amsart}

\usepackage{amssymb}

\usepackage{enumitem}

\usepackage{graphicx}

\makeatletter
\@namedef{subjclassname@2010}{%
	\textup{2010} Mathematics Subject Classification}
\makeatother

\usepackage[T1]{fontenc}
\newtheorem{theorem}{Theorem}[section]
\newtheorem{corollary}[theorem]{Corollary}
\newtheorem{lemma}[theorem]{Lemma}
\newtheorem{proposition}[theorem]{Proposition}

\theoremstyle{definition}
\newtheorem{definition}[theorem]{Definition}
\newtheorem{remark}[theorem]{Remark}
\newtheorem{example}[theorem]{Example}

\numberwithin{equation}{section}

\begin{document}
	
	%%%%% To ease editing, for IMPAN journals add:
	
	\baselineskip=17pt
	
	%%%%%%%%%%%%%%%%
	
	\title{A note on the   orbit equivalence of injective actions}

	\author[X. Q. Qiang]{Xiangqi Qiang}
	\address{1.School of Science \\ Jiangsu University of Science and Technology\\
		Zhenjiang 212100, China}
			\address{2.School of Mathematical Science \\ Yangzhou University\\
			Yangzhou 225002, China}
	\email{xq.qiang@just.edu.cn}
	
	\author[C. J. Hou]{Chengjun Hou}
	\address{School of Mathematical Science \\ Yangzhou University\\
		Yangzhou 225002, China}
	\email{cjhou@yzu.edu.cn}

	\begin{abstract}
	 	We characterise the groupoid $C^*$-algebras associated  to  the transformation groupoids of  injective actions of discrete countable Ore semi-groups on compact topological  spaces  in terms of  the   reduced crossed product from the  dual actions, and characterise the continuous orbit equivalence for injective actions by means of  the   transformation groupoids, as well as   their  reduced groupoid $C^*$-algebras. Finally,  we characterize  the   injective  action  of  semi-group on its   compactifications.
		
	\end{abstract}
	
	
	
	\maketitle
	
\section{Introduction}
There are a large number of significant and interesting research dealing with    the interplay between the orbit equivalence of topological dynamical systems and the classification of $C^{*}$-algebras.  Groupoid theory and   	the crossed product construction play a very important role for these results. At the very beginning,  Giordano, Putnam and Skau studied the relationship between   orbit equivalence and $C^{*}$-crossed products for   minimal homeomorphisms of Cantor sets in \cite{article.5}. Their studies have  been    generalized to many different directions, including Tomiyama's results on topologically free homeomorphisms on compact Hausdorff spaces (\cite{article.15}), Matsumoto et al.'s classification results of irreducible  topological Markov shifts in terms of Cuntz-Krieger algebras (\cite{article.10,article.11}), and Li's  characterization of  group actions and partial dynamical systems by transformation groupoids and  their reduced crossed product $ C^{*} $-algebras (\cite{article.8,article.9}), and so on (\cite{article.2,article.6,article.12}).

In  \cite{article.14}, Renault and   Sundar studied actions of locally compact Ore semi-groups  on compact topological spaces. They   gave the construction of the semi-direct product	$X\rtimes P$, where $X$ is   the order compactification of  locally compact   semi-group $ P $, and turned out that it is a locally compact groupoid  and has a continuous Haar system. They also proved that 	this groupoid is  a reduction of a semi-direct product by a group and  the   Wiener-Hopf $C^*$-algebra of   $ P $ is isomorphic to the reduced $C^*$-algebra of the semi-direct product groupoid $X\rtimes P$. In \cite{article.4}, Ge studied the compactification of natural numbers and  characterized the associated compact Hausdorff spaces. Inspired by their beautiful results, in this paper,  we will consider    injective actions of discrete countable Ore semi-groups on compact topological  spaces, and study the relationship between the orbit structure of   these actions  and algebraic structure of the associated  groupoids and their  $C^*$-algebras. We prove that two topologically free injective actions are continuously orbit equivalent if and only if their transformation groupoids are isomorphic as \'{e}tale groupoids, if and only if there is a $C^{*}$-isomorphism  preserving the canonical Cartan subalgebras between the corresponding groupoid reduced $C^{*}$-algebras.  We   show that an  injective action can be dilated as a group action on  a quotient space by homeomorphisms, and  our   groupoid turns out to be a reduction of the dilation.  The multiplication operation on the group can  naturally give rise to an  action of the  group on itself, thus  we consider the induced  action of the semi-group on its  compactifications. We also show that the injective action of a semi-group on its one-point  compactification is determined uniquely up to conjugacy by two conditions.

The paper is organized as follows. In Section  2, we  characterize the reduced groupoid $C^*$-algebras  associated to the transformation groupoids of injective actions of discrete countable Ore semi-groups on compact topological  spaces  in terms of the   reduced crossed product from the  dual actions.    In Section 3, we introduce the notion of  continuous orbit equivalence for injective actions,  and characterize them in terms of  the associated transformation groupoids, as well as   their reduced groupoid $C^{*}$-algebras with canonical Cartan subalgebras. In  Section 4, we study the action of semi-group on its  compactifications.

Throughout this paper, we will use the following notions.  For a topological  groupoid $\mathcal{G}$, let $\mathcal{G}^{(0)}$   be  the unit space. The range and   source maps $r,s$ from $\mathcal{G}$ onto $\mathcal{G}^{(0)}$ are defined by $r(g)=gg^{-1}$ and $s(g)=g^{-1}g$, respectively.
If $r$ and $s$ are local homeomorphisms, then $\mathcal{G}$ is called to be \'{e}tale. We say that $\mathcal{G}$ is \emph{topologically principal} if $\left\lbrace u\in \mathcal{G}^{(0)}:\,\mathcal{G}_{u}^{u}=\left\lbrace u\right\rbrace \right\rbrace $ is dense in $\mathcal{G}^{(0)}$, where $\mathcal{G}_{u}^{u}= \left\lbrace\gamma\in \mathcal{G}:  r(\gamma) =s(\gamma)=u\right\rbrace   $ is the isotropy group at a unit $u\in \mathcal{G}^{(0)}$.
We refer to \cite{book.1,book.2} for more details on topological groupoids and their $C^*$-algebras.

\section{The transformation groupoid of injective action}

Let $X$ be a second-countable compact Hausdorff space, $G$ a countable discrete group and $P$ a right Ore sub-semigroup of $G$, i.e., $e\in P$ and  $G=PP^{-1}$, where $e$ is the identity element in $G$.  By  a right action $\theta$ of $P$ on $X$ we mean that  $\theta_{a}$ is a continuous and injective map from $X$ into itself  and   satisfies that $\theta_{a}\theta_{b}=\theta_{ba}$ for every $a,b \in P$, and $\theta_e=id_{X}$, the identity map on $X$.   We use the symbol $ P\curvearrowright_{\theta} X $  to denote such an injective   action.

Each injective action $ P\curvearrowright_{\theta} X $ induces a dual action, $ P\curvearrowright_{\alpha} C(X)$, of $P$ on the abelian $C^*$-algebra $C(X)$ by surjective $\ast$-homomorphisms, where each  $\alpha_m$ is defined by $\alpha_{m}(f) =f\theta_{m} $ for $f\in C(X)$ and  $\alpha_{ab}=\alpha_{a}\alpha_{b}$ for all $a$ and $b$ in $P$.

\begin{remark}	
	Given an injective  action  $ P\curvearrowright_{\theta}  X $,
	for $x\in X$, let $$Q_{x}:=\{g\in G: \exists a,b\in P,\, y\in X\; \text{such that} \; g=ab^{-1} \; \text{and} \; \theta_{a}(x)=\theta_{b}(y)\}.$$
	For  $x,y\in X$ and   $g\in G$, one can check that, if $g=ab^{-1}=mn^{-1}$ for $a,b,m,n\in P$, then $\theta_{a}(x)=\theta_{b}(y)$ if and only if $\theta_{m}(x)=\theta_{n}(y)$.
	It follows that
	$g\in Q_{x}$ if and only if there exists a unique element, denoted by $u(x,g)$, in $ X$ such that if $g=ab^{-1}$ for $a,b\in P$, then $\theta_{a}(x)=\theta_{b}(u(x,g))$.
\end{remark}

Let $$ X\rtimes P=\{(x,g)\in X\times G : g\in Q_x\}. $$
Then, under the following operations
$$(x,g)(y,h)=(x,gh)\;\text{only if}\;y=u(x,g),$$
$$(x,g)^{-1}=(u(x,g),g^{-1}),$$
$X\rtimes P $ is a groupoid with the unit space $(X\rtimes P)^{(0)}=X\times \{e\}$, the range map, $r(x,g)=(x,e)$ and the source map $s(x,g)=(u(x,g),e)$ for $(x,g)\in X\rtimes P$.

\begin{lemma}
	The  map $u: (x,g)\in X\rtimes P \rightarrow u(x,g)\in X$ is continuous, where $u(x,g)$ is defined as in Remark 2.1. Moreover, $u(x,m)=\theta_m(x)$ for $x\in X$ and $m\in P$, and $u(u(x,g),h)=u(x,gh)$ when $(x,g), (u(x,g),h)\in X\rtimes P$.
	
	Thus, under the relative product topology  on $X\times G$, $X\rtimes P$ is a second-countable locally compact Hausdorff groupoid. Furthermore, $X\rtimes P$ is \'{e}tale if and only if $\theta_{a}(X)$ is open in $X$ for each $a\in P$.
	
\end{lemma}
\begin{proof}
	
	Suppose that $(x_{n},g_{n})\rightarrow (x,g)\in X\rtimes P$. Then $g_{n}=g$ for large $n$, so we can assume that $g_{n}=g$ for all $n$. Let $y_{n}=u(x_{n},g)$ and suppose that $y_{n}\rightarrow y$. Choose $a,b\in P$ such that $g=ab^{-1}$ and $\theta_{a}(x_{n})=\theta_{b}(y_{n})$. Then   $\theta_{a}(x)=\theta_{b}(y)$, it follows that $y=u(x,g)$, proving that $u$ is continuous.

	Given $(x,g)$ and $(u(x,g),h)$ in $X\rtimes P$, choose $a,b,c,d,m$ and $n$ in $ P$ such that $g=ab^{-1}$, $h=cd^{-1}$ and $b^{-1}c=mn^{-1}$. Then $bm=cn$ and $gh=am(dn)^{-1}$. It follows from $\theta_{a}(x)=\theta_{b}(u(x,g))$ and $\theta_{c}(u(x,g))=\theta_{d}(u(u(x,g),h))$ that $\theta_{am}(x)=\theta_{dn}(u(u(x,g),h))$. Since moreover $\theta_{am}(x)=\theta_{dn}(u(x,gh))$, we have $ u(u(x,g),h)= u(x,gh)$.
	
	By the continuity of $u$, one checks that $X\rtimes P$ is a second-countable locally compact Hausdorff groupoid.  If $  X\rtimes P $ is  \'{e}tale, then the   source map $s: X\rtimes P\rightarrow  (X\rtimes P)^{(0)}, (x,g)\mapsto (u(x,g),e)$ is open. It follows that, for each $a\in P$, $s(X\times\{a\})=\theta_{a}(X)\times \{e\}$ is open in $  X\rtimes P $. Thus $\theta_{a}(X)$ is open in $X$ for each $a\in P$.
	
	For the converse, assume that $\theta_{a}(X)$ is open in $X$ for each $a\in P$. Then $\theta_{a}: X\rightarrow \theta_{a}(X)$ is a homeomorphism.
	For any $(x_{0},g_{0})\in  X\rtimes P $, there exist $m,n\in P$, $y_{0}\in X$ such that $g_{0}=mn^{-1}$ and $\theta_{m}(x_{0})=\theta_{n}(y_{0})$. Since $\theta_{m}(X)$  and  $\theta_{n}(X)$ are open in $X$, there exist open subsets $W_{1}\subseteq \theta_{m}(X)$ and $W_{2}\subseteq\theta_{n}(X)$ such that $\theta_{m}(x_{0})\in W_{1}$, $\theta_{n}(y_{0})\in W_{2}$. Set $W=W_{1}\cap W_{2}$, $U= \theta_{m}^{-1}(W)$, and $V= \theta_{n}^{-1}(W)$. Thus $ \theta_{m}(x_{0})=\theta_{n}(y_{0})\in W$, $U,V$ are open in $X$ and $ \theta_{m}(U)=\theta_{n}(V)=W$. In this case, for each $x\in U$, there exists $y\in V$ such that $\theta_{m}(x)=\theta_{n}(y)$,  it follows that $(x,g_{0})\in X\rtimes P $ for each $x\in U$, which implies $ U\times\{g_{0}\}$ is an open neighbourhood of $(x_0,g_0)$ in $X\rtimes P $.
	Thus the range map $r|_{U\times\{g_{0}\}}: U\times\{g_{0}\}\rightarrow U\times\{e\}$ is a homeomorphism. Hence $  X\rtimes P $ is  \'{e}tale.
	
\end{proof}

\begin{remark}\label{R1} Throughout this paper, we always assume that $\theta_{a}(X)$ is open in $X$ for each $a\in P$,and  we call $  X\rtimes P $ the transformation groupoid attached to $ P\curvearrowright_{\theta} X $. In this case, the unit space $ (X\rtimes P)^{(0)} $ identifies with $X$ by identifying $(x,e)$ with $e$. Thus $r(x,g)=x$ and $s(x,g)=u(x,g)$. From the proof of the last lemma, for each $(x,g)\in  X\rtimes P$, there exists an open neighbourhood $U$ of $x$ in $X$ such that $(x,g)\in U\times\{g\}\subseteq X\rtimes P$.
	
	Let $c:\;X\rtimes P\rightarrow G$ be  defined by $c(x,g)=g$ for $(x,g)\in X\rtimes P$. Then $c$ is a continuous homomorphism and its kernel $ker(c)=X\times\{e\}$ is an amenable \'{e}tale subgroupoid of $X\rtimes P$. It follows from \cite[Proposition 10.1.11]{book.2} that if $G$ is amenable then $X\rtimes P$ is also amenable.
	
\end{remark}

In \cite{article.14}, the transformation groupoid is isomorphic to a reduction of the Mackey range semi-direct product defined by the canonical cocycle $c$ when $P$ is a locally compact Ore semi-group. For the countable discrete case, in the following we can give a direct construction of this result.

Let  $\widetilde{X}  $ be the quotient space of $X\times G$  by the following equivalence relation:
$$(x,g)\sim(y,h)\Leftrightarrow \exists a,b\in P\;\text{such that}\;gh^{-1}=ab^{-1}\;   \text{and }\; \theta_{a}(x)=\theta_{b}(y).$$
One can check that, under the quotient topology of the product topology on $X\times G$, $\widetilde{X}$ is a locally compact and Hausdorff space, and the canonical quotient map  $\pi: X\times G\rightarrow \widetilde{X} $ is surjective, continuous and open.
Denote by $[x,g]$ the equivalence class of $(x,g)$ in the equivalence relation. Then $[x,g]=\{ (u(x,k) ,k^{-1}g)):k\in Q_{x}\}$ for $(x,g)\in X\times P$, and $[x,e] =[u,e] $ if and only if $x=u$.

\begin{remark}\label{R2} 	Define the right  action $ G\curvearrowright_{ \tilde{\theta}}\widetilde{X} $  of $G$ on $\widetilde{X}$ by homeomorphisms as follows: for $ [x,g] \in\widetilde{X}, h\in G$, $$\tilde{\theta}_{h} ([x,g] )=[x,gh]  .$$
	The associated transformation groupoid, $\widetilde{X}\rtimes_{\tilde{\theta}}G: =\widetilde{X}\times G$, with the product topology and the following   multiplication and inverse:
	$$([x,g],h)([x,gh],h')=([x,g],hh'),\;([x,g],h)^{-1}=([x,gh],h^{-1}),$$ is an \'{e}tale groupoid.
	The unit space $(\widetilde{X}\rtimes_{\tilde{\theta}}G)^{(0)}$ identifies with  $\widetilde{X}$ by identifying $([x,g],e)$ with $[x,g]$.  Then $r([x,g],h)=[x,g]$  and $s([x,g],h)=[x,gh]$.
	
	Let $X'=\{[x,e]:x\in X\}$. Then $X'$ is clopen in $\widetilde{X}$ and $ \widetilde{X}\rtimes_{\tilde{\theta}}G $-full in the sense that the intersection of $X'$ and each orbit of $ G\curvearrowright_{ \tilde{\theta}}\widetilde{X} $ is not empty. In fact, for any $[x,g]\in \widetilde{X} $,
	$ ([x,g],g^{-1})\in r^{-1}([x,g] )\cap s^{-1}(X') $. Note that the reduction    $  \widetilde{X}\rtimes_{\tilde{\theta}}G|_{X'}=r^{-1}(X')\cap s^{-1}(X')$    is an \'{e}tale subgroupoid of $\widetilde{X}\rtimes_{\tilde{\theta}}G$. 	Recall that two \'{e}tale groupoids $\mathcal{G}$ and $\mathcal{H}$ are Kakutani equivalent if there are full clopen subsets $X\subseteq\mathcal{G}^{(0)}$ and $Y\subseteq\mathcal{H}^{(0)}$ such that $ \mathcal{G}|_{X}\cong \mathcal{H}|_{Y} $ (\cite{article.3}). The   following proposition is based on the ideas of  \cite{article.14} by   Renault and   Sundar.

\end{remark}

\begin{proposition}
	$ X\rtimes P $ is isomorphic to the reduction $ \widetilde{X}\rtimes_{\tilde{\theta}}G|_{X'}  $. Consequently, $ X\rtimes P $ is Kakutani equivalent to $ \widetilde{X}\rtimes_{\tilde{\theta}}G$.
\end{proposition}
\begin{proof}
	Note that   $r([x,e],g)=[x,e]\in X'$ and $s([x,e],g)=[x,g]=[u(x,g),e]\in X'$ for each $(x,g)\in X\rtimes P$. We can therefore define a map  $$ \Phi:X\rtimes P \rightarrow \widetilde{X}\rtimes_{\tilde{\theta}}G|_{X'}, \;(x,g) \mapsto ([x,e],g),$$ and $\Phi$ is clearly an injective groupoid homomorphism. For a given $(y,g)\in \widetilde{X}\rtimes_{\tilde{\theta}}G|_{X'}$, since $r(y,g)=y\in X'$,   there exists $x\in X$ such that $y=[x,e]$. In this case, $s([x,e],g)=[x,g]\in X'$, which implies $(x,g)\in X\rtimes P$, proving that $\Phi$ is surjective.

	Let $  \varphi $ be the restriction of $\Phi$ to the unit space $(X\rtimes P )^{(0)}=X $,
	one can check that $\varphi: x\in X\rightarrow [x,e]\in X'$ is a homeomorphism. It is then easy to see that $ \Phi$ is a homeomorphism from $ X\rtimes P $ onto $  \widetilde{X}\rtimes_{\tilde{\theta}}G|_{X'} $.

\end{proof}

\begin{remark}\label{R3}
	Given an injective right action $P\curvearrowright_{\theta}X$, we further assume that each map $\theta_{m}$ is a homeomorphism on $X$. For each $g\in G$, it follows from the assumption that there exist $m,n\in P$ such that $g=mn^{-1}$.
	Define
	$$\hat{\theta}_{g}(x)=\theta_{n}^{-1}(\theta_{m}(x)) \;\;\text{for}\;  x\in X.$$
	One can check that $  \hat{\theta}  $ is a right    action of $G$ on $X$ by homeomorphisms. In this case, the equivalence relation on $X\times G$ defined before Remark 2.4 can be rewritten as follows: for $ (x,g),(y,h)\in X\times G$,
	$$\begin{array}{lll}
		(x,g)\sim  (y,h)
		\Leftrightarrow y=\hat{\theta}_{gh^{-1}}(x).
	\end{array}$$
	Thus $[x,g] =[\hat{\theta}_{g}(x),e] $ for $[x,g]\in \widetilde{X}$, and  $ \widetilde{X}=X'$.
\end{remark}

The transformation groupoid  $X\rtimes_{\hat{\theta}} G$ associated to the above group action $(X,G,\hat{\theta})$   is given by the set $X\times G$ with the product topology,  multiplication $(x,g)(y,h)=(x,gh)$ if $y=\hat{\theta}_g(x)$, and inverse $(x,g)^{-1}=(\hat{\theta}_g(x),g^{-1})$.

Remark \ref{R2} and  Remark \ref{R3} combine to give the following result.
\begin{corollary}
	If $P\curvearrowright_{\theta}X$ is a right action by homeomorphisms, then $ X\rtimes P $ is isomorphic to $ \widetilde{X}\rtimes_{\tilde{\theta}}G   $, and  both of them are isomorphic to     $X\rtimes_{\hat{\theta}}  G  $.
\end{corollary}

From Proposition 2.5 and Corollary 2.7, the reduced groupoid $C^*$-algebra $C_r^*(X\rtimes P)$ of $X\rtimes P$ is Morita equivalent to the reduced crossed product $C^*$-algebra $C(\widetilde{X})\rtimes_{\widetilde{\theta}} G$ associated to $ G\curvearrowright_{ \tilde{\theta}}\widetilde{X} $, and these two $C^*$-algebras are isomorphic when $\theta$ is a homeomorphism action. In the rest of this section, we characterize  the relationship among $C_r^*(X\rtimes P)$,  the $C^*$-algebra from the dual action $P\curvearrowright_{\alpha} C(X)$ and partial action of $G$ on $C(X)$ given by  $P\curvearrowright_{\theta} X$.

We define $l^2(G,C(X))$ to be the set of all mapping $\xi$ from $G$ into $C(X)$ such that $\sum_{g\in G}|\xi(g)|^2$ converges in $C(X)$. Then it is a (right) Hilbert $C(X)$-module under the following operations:
$$(\xi f)(g)=\xi(g)f,\,\, <\xi,\eta>=\sum_{g\in G}\xi(g)^{\ast}\eta(g)$$
for  $f\in C(X), \xi,\eta\in l^2(G,C(X)), g\in G$. Similarly, we have the  Hilbert $C(X)$-module $E:\,=l^2(P,C(X))$ and let $\mathcal{L}(E)$ be the $C^{*}$-algebra of all adjointable operators on $E$.

Define representations
$\pi:\, f\in C(X)\rightarrow \pi(f)\in \mathcal{L}(E)$  and $v:\, m\in P\rightarrow v_m\in \mathcal{L}(E)$ by
$$\pi(f)(\xi)(m):=\alpha_{m}(f)\xi(m),\,\,\, v_m(\xi)(n):=\widetilde{\xi}(nm^{-1})\,\,\,\mbox{ for $\xi\in \mathcal{L}(E)$, $m,n\in P$},$$
where $\widetilde{\xi}\in l^2(G,C(X))$ is given by $\widetilde{\xi}(g)=\left\{\begin{array}{ll} \xi(g), & \mbox{ for $g\in P$}\\ 0, & \mbox{ for otherwise} \end{array}\right.$  for $g\in G$.
Then  $v_m^*\xi(n)=\xi(nm)$ for $m,n\in P$, $\xi\in \mathcal{L}(E)$. One can check that $v_e=I$, $v_m$ is an isometry, $v_mv_n=v_{nm}$ and $\pi(f)v_m=v_m\pi(\alpha_m(f))$ for $f\in C(X), m, n\in P$.
The $C^*$-algebra generated by $\{\pi(f), v_m:\, f\in C(X), m\in P\}$ in $\mathcal{L}(E)$ is the reduced crossed product associated with $P\curvearrowright_{\alpha} C(X)$, denoted by $C(X)\rtimes_{r} P$.

To simplify symbol, write $\mathcal{G}=X\rtimes P$. Let $l^{2}(\mathcal{G})$ be a Hilbert right $C(X)$-module by the completion of $C_{c}(\mathcal{G})$ under the following operations:
$$(\xi f)(x,g)=\xi(x,g)f(x)\; \mbox{ for }\xi\in C_{c}(\mathcal{G}), f\in C(X), (x,g)\in \mathcal{G}.$$
$$<\xi,\eta>(x)=\sum_{g\in G,(x,g)\in \mathcal{G}}   \overline{\xi(x,g)} \eta(x,g) \; \mbox{ for } \xi,\eta\in C_{c}(\mathcal{G}) ,x\in X.$$
$$\|\xi\|=\sup_{x\in X} ( \sum_{g\in G,(x,g)\in \mathcal{G}}|\xi(x,g)|^{2} ) ^{\frac{1}{2}} \; \mbox{ for } \xi \in C_{c}(\mathcal{G}).$$

Let $ \mathcal{L}(l^{2}(\mathcal{G})) $ be the $C^{*}$-algebra of all adjointable operators on $l^{2}(\mathcal{G})$. Define the representation $\widetilde{\pi}$ of $C_{c}( \mathcal{G})$ into $ \mathcal{L}(l^{2}(\mathcal{G})) $ by $\widetilde{\pi}: f\in C_{c}( \mathcal{G})\rightarrow\widetilde{\pi}(f)\in \mathcal{L}(l^{2}(\mathcal{G}))$:
$$(\widetilde{\pi} ( f)\xi)(x,g)= \sum_{ (x,h)\in \mathcal{G}}f(u(x,g),g^{-1}h)\xi(x,h) \; \mbox{ for } (x,g)\in \mathcal{G}.$$
Then the reduced groupoid $C^*$-algebra, $ C^{*}_{r}( \mathcal{G}) $, of $\mathcal{G}$ is the $C^*$-algebra generated by $\widetilde{\pi}( C_{c}( \mathcal{G}))$ in $ \mathcal{L}(l^{2}(\mathcal{G}))$ (\cite{article.1}).

\begin{lemma} \label{L1}
	Let $ P\curvearrowright_{\theta} X $ be an injective action. For $g\in G$, let $X_g=\{x\in X:\, (x,g)\in X\rtimes P\}$ and $U_g=X_g\times \{g\}$. Then the characteristic function, denoted by $u_g$, on  $U_g$ is in $C_c(X\rtimes P)$. Moreover, the following statements hold:
	\begin{enumerate}
		\item[(i)] $u_e$ is the identity element in $C_c(X\rtimes P)$, $u_a^*$ is an isometry and $u_au_b=u_{ab}$ for $a,b\in P$;
		\item[(ii)] for $g\in G$, if $g=ab^{-1}$ for $a,b\in P$, then $u_g=u_au_b^*$ and  $u_g^*=u_{g^{-1}}$;
		\item[(iii)] for $f\in C(X)$ and $g\in G$, we have $u_gf=V_g(f) u_g$, where for $x\in X$,
		
		$V_g(f)(x)=\left\{\begin{array}{ll} f(u(x,g)), & \mbox{ if $x\in X_g$}
			\\ 0 , & \mbox{ for otherwise.}\end{array}\right.$
		
		\item[(iv)] $u_gu_{g^{-1}}=\chi_{X_g}\in C(X)$ and $u_gfu_{g^{-1}}=V_g(f)u_gu_{g^{-1}}$.
	\end{enumerate}
	
	Hence $C_c(X\rtimes P)=span\{fu_g:\,\, f\in C(X), g\in G\}$.
	
\end{lemma}

\begin{proof} Note that for $g\in G$, $U_g$ is an open and compact subset of $X\rtimes P$, thus $u_g\in C_c(X\rtimes P)$. By calculation, we can check the properties stated in the lemma. For $\xi\in C_c(X\rtimes P)$, there exist $g_1,g_2, \cdots, g_n\in G$ such that the support $supp(\xi)$ of $\xi$ is contained in $\cup_{i=1}^nU_{g_i}$. By a partition of unity, we have $\xi=\sum_{i=1}^n\xi_i$ for $\xi_i\in C_c(X\rtimes P)$ and $supp(\xi_i)\subseteq U_{g_i}$. Let $$f_i(x)=\left\{\begin{array}{ll} \xi_i(x,g_i), & \mbox{ if $x\in X_{g_i}$}
		\\ 0, & \mbox{ for otherwise}\end{array}\right.$$ for $x\in X$. Then $f_i\in C(X)$ and $\xi_i=f_iu_{g_i}$ for each $i$. Thus $C_c(X\rtimes P)=span\{fu_g:\,\, f\in C(X), g\in G\}$.

\end{proof}

\begin{theorem}\label{T1}
	Let $\mathcal{M}$ be the closure of the set $\{\xi\in C_c(X\rtimes P):\, \xi(x,g)=0 \mbox{ if $g\notin P$}\}$ in $l^2(X\rtimes P)$ and $Q$ be the projection from $l^2(X\rtimes P)$ onto $\mathcal{M}$. Then $\mathcal{M}$ is isomorphic to $l^2(P, C(X))$, and $C(X)\rtimes_{r}P$ is isomorphic to the $C^*$-algebra generated by $QC_r^*(X\rtimes P)Q$ in $\mathcal{L}(l^2(X\rtimes P))$.
\end{theorem}

\begin{proof} We use the notation in Lemma 2.8 and before. Obviously, $\mathcal{M}$ is a (right) $C(X)$-submodule of $
	l^{2}(\mathcal{G} )$. Define
	$$\Lambda: \mathcal{M}\rightarrow l^2(P, C(X)),\;\; \Lambda(\zeta)(m)(x)=\zeta(x,m),$$ for $\zeta\in \mathcal{M}$, $m\in P$ and $x\in X$. Then $\Lambda$ is a bijective bounded $C(X)$-linear mapping with inverse 	\begin{equation*}
		\Lambda^{-1}(\varepsilon)(x,g)=\begin{cases}  \varepsilon(g)(x),& \text{if}\; g\in P, \\0, & \text{for otherwise},\end{cases} \text{for }\;\varepsilon\in l^2(P, C(X)), (x,g)\in \mathcal{G}.
	\end{equation*}
	Moreover, for $ \zeta_{1},\zeta_{2}\in  \mathcal{M} $, one can check that $<\zeta_{1},\zeta_{2}>(x)=<\Lambda\zeta_{1},\Lambda\zeta_{2}>(x)$ for each $x\in X$, then 	$\mathcal{M}$ is isomorphic to $l^2(P, C(X))$.
	
	Define
	$$W(\eta)(m)(x)=\eta(x,m),\;\text{for}\;\eta\in C_{c}(\mathcal{G}),  m\in P,x\in X,$$ and
	\begin{equation*}
		U(\xi)(x,g)=\begin{cases}  \xi(g)(x),& \text{if}\; g\in P, \\0, & \text{for otherwise},\end{cases} \text{for }\;\xi\in  C_{c}(P,C(X)), (x,g)\in \mathcal{G}.
	\end{equation*}
	
	Then $W$ and $U$ can be extended to operators in $\mathcal{L}(l^{2}(\mathcal{G} ),l^{2}(P,C(X))) $  and  $ \mathcal{L}(l^{2}(P,C(X)),l^{2}( \mathcal{G}) ) $, and if we use the same symbols to denote their extensions then $ U^{*}=W $ and $U^*U=id$, the identity element in $ \mathcal{L}(l^{2}(P,C(X))) $. An easy calculation confirms that   $ U\pi(f)=\widetilde{\pi}(f)U $ for $f\in C(X)$, and $Uv_m=\widetilde{\pi}(u_{m}^*)U$ for   $m\in P$.
	
	Write $Q=UU^*$. Define the map $\Phi: \mathcal{L}(l^{2}(P, C(X)))\rightarrow \mathcal{L}(l^{2}(\mathcal{G}))$ by  $$\Phi(T)=UTU^*.$$
	Then $\Phi$ is an injective $\ast$-homomorphism, and $\Phi(\pi(f))=Q\widetilde{\pi}(f)Q$, $\Phi(v_m)=Q\widetilde{\pi}(u_m^*)Q$ for $f\in C(X)$, $m\in P$. One can check that $Q\widetilde{\pi}(u_m)\widetilde{\pi}(u_n^*)Q=Q\widetilde{\pi}(u_m)Q\widetilde{\pi}(u_n^*)Q$,  $Q\widetilde{\pi}(f)\widetilde{\pi}(u_g)Q=Q\widetilde{\pi}(f)Q\widetilde{\pi}(u_g)Q$ for each $f\in C(X)$, $m,n\in P$ and $g\in G$. Thus it follows from Lemma 2.8 that $\Phi(C(X)\rtimes P)$ is just the $C^*$-algebra generated by $QC_r^*(X\rtimes P)Q$ in $\mathcal{L}(l^{2}(\mathcal{G}))$.

\end{proof}

We adopt notations in Lemma 2.8 and define a mapping $\hat{\alpha}_{g}: C(X_{g^{-1}})\rightarrow C(X_{g})$ as follows:     $$\hat{\alpha}_{g}(f)(x)=f(u(x,g)),\;\;  \text{for}\; f\in C(X_{g^{-1}}) ,  x\in X_{g}.$$
The preceding descriptions imply that $\hat{\alpha}_{g}$ is an isomorphism. Then $\{\hat{\alpha}_{g}\}_{g\in G}$ defines a partial action of $G$ by partial isomorphism of $C(X)$, and $(C(X), G, \hat{\alpha}) $ is a partial $C^{*}$-dynamical system in the sense of \cite{article.7}.

Consider now the Hilbert (right) $C(X)$-submodule $F=\{ \xi\in l^{2}(G,C(X)):\xi(g)\in C(X_{g})\}$. Denoted by $\mathcal{L}(F)$ be the $C^{*}$-algebra of all adjointable operators on $F$. Define the representation $ \tau:f\in C(X) \rightarrow\tau(f)\in \mathcal{L}(F)$ and $v:g\in G\rightarrow v_{g}\in \mathcal{L}(F)$ by
$$\tau(f)\xi(g)=\hat{\alpha}_{g}(f)\xi(g)|_{X_{g}},\; v_{h}\xi(g)=\xi(gh)|_{X_{g}\cap X_{gh}}, \;\text{for}\;\xi\in \mathcal{L}(F), g,h\in G.$$ Then $v_{g}$ is a partial isometry on $F$ with initial  space $[\tau(C(X_{g^{-1}}))F]$ and final  space $[\tau(C(X_{g}))F]$ such that
\begin{enumerate}
	\item[(i)] $v_{g}\tau(f)v_{g^{-1}}=\tau(\hat{\alpha}_{g}(f)) $ for $f\in C(X_{g^{-1}})$;
	\item[(ii)] $ \tau(f)[v_{g}v_{h}-v_{gh}] =0$ for $f\in C(X_{g})\cap C(X_{gh})$;
	\item[(iii)] $ v_{g}^{*}=v_{g^{-1}} $.

\end{enumerate}
Hence $(\tau,v,F)$ is a covariant representation of $(C(X), G, \hat{\alpha}) $.

The  reduced partial crossed product, denoted by  $C(X)\rtimes^{\widehat{\alpha}}_{r}G$, associated with $(C(X), G, \hat{\alpha}) $ is defined as the $C^*$-algebra generated by $\{\tau(f), v_g:\, f\in C(X), g\in G\}$ in $\mathcal{L}(F)$.

\begin{theorem}
	$C^{*}_{r}(  X\rtimes P)  $ is isomorphic to $C(X)\rtimes^{\widehat{\alpha}}_{r}G$.
\end{theorem}

\begin{proof}
	Define 	\begin{equation*} \Phi(\eta)(g)(x)=\begin{cases} \eta (x,g) & \text{if}\; x\in X_{g},\\ 0,& \text{for otherwise},\end{cases} \text{for }\;\eta\in  C_{c}(\mathcal{G}),g\in G, \text{and}\;  x\in X.\end{equation*}
	Then $\Phi$ can be extended to operator in $\mathcal{L}(l^{2}(\mathcal{G}), F)$, and we use the same symbol to denote  its extension. Moreover, $\Phi$ is an adjointable unitary operator in $\mathcal{L}(l^{2}(\mathcal{G}), F)$ with $\Phi^{*}(\xi)(x,g)=\xi(g)(x)$ for  $\xi\in F,(x,g)\in \mathcal{G}$.

	Define $\Psi: T\in \mathcal{L}(l^{2}(\mathcal{G}))\rightarrow \Psi(T)\in  \mathcal{L}(F)$ by $$\Psi(T)=\Phi T\Phi^{*}.$$ Then $\Psi$  is an $*$-isomorphism, and by calculation, $\Psi(\widetilde{\pi}(f))=\tau(f),\Psi(\widetilde{\pi}(u_{g}))=v_{g}$, thus $C^{*}_{r}(  X\rtimes P)  $ is isomorphic to $C(X)\rtimes^{\widehat{\alpha}}_{r}G$.

\end{proof}

\section{ Continuous orbit equivalence}
Let $ P\curvearrowright_{\theta} X $ be an injective action and $ G $ a countable group containing
$ P $ as in Section 2. Define $$ x\sim_{\theta} y \Leftrightarrow \exists g\in Q_{x}\; \text{such that}\; y=u(x,g).$$ Then $  \sim_{\theta}   $ is an equivalent relation on $X$.  We denote by $[x]_{\theta} $ the equivalence class of $x$, i.e., $ [x]_{\theta} :=\{u(x,g):g\in Q_{x}\} .$

Given two injective actions $  P\curvearrowright_{\theta} X  $ and  $S\curvearrowright_{\rho} Y $, we let $ G $ and $ H $ be two related discrete groups satisfying that $ P  \subseteq G $, $ S \subseteq H  $ and    the assumption in
Section 2.

\begin{definition} Let  $   P\curvearrowright_{\theta} X  $ and  $S\curvearrowright_{\rho} Y  $  be two injective actions.
	\begin{enumerate}
		\item[(i)]
		We say they   are \emph{ conjugate} if there exist a homeomorphism $\varphi:\, X\rightarrow Y $ and a semi-group isomorphism $\alpha:P\rightarrow S$ such that $\varphi\theta_{m}= \rho_{\alpha(m)}\varphi$ for each $m\in P$.
		\item[(ii)]
		We say they   are \emph{ orbit equivalent} if there exists a homeomorphism $\varphi:\, X\rightarrow Y $ such that $\varphi([x ]_{\theta}) =[\varphi(x) ]_{\rho}$ for $x\in X$.
	\end{enumerate}
\end{definition}

If $   P\curvearrowright_{\theta} X  $ and  $S\curvearrowright_{\rho} Y  $ are  orbit equivalent via a homeomorphism $ \varphi $,
then for each $x\in X, g\in Q_{x}$, there exists $h\in Q_{\varphi(x)}$ (depending on $x$ and $g$ ) such that $ 	\varphi(u(x,g))=u(\varphi(x),h) $. Symmetrically, for each $y\in Y, h\in Q_{y}$, there exists $g\in Q_{ \varphi^{-1}(y)} $ (depending on $y$ and $h$) such that  $ 	\varphi^{-1}(u(y,h))=u(\varphi^{-1}(y),g) $.  Note  that $\cup_{x\in X}\{x\}\times Q_{x}=X\rtimes P$,   we have the following continuous version of
orbit equivalence.
\begin{definition}\label{De1}  We say two injective actions  $   P\curvearrowright_{\theta} X  $ and  $S\curvearrowright_{\rho} Y $  are \emph{continuously orbit equivalent} if there exist  a homeomorphism $\varphi:\, X\rightarrow Y $, continuous mappings
	$a :\; X\rtimes   P\rightarrow H$    and  $ b  :\; Y\rtimes  S\rightarrow G $ such that
	\begin{equation}\label{e1}
		\varphi(u(x,g))=u(\varphi(x),a(x,g)) \;\; \text{for} \; (x,g)\in X\rtimes   P,
	\end{equation}
	\begin{equation}\label{e2}
		\varphi^{-1}(u(y,h))=u(\varphi^{-1}(y),b(y,h))\;\; \text{for} \;  (y,h)\in  Y\rtimes  S.
	\end{equation}

\end{definition}
\begin{proposition}
	If two injective actions   $   P\curvearrowright_{\theta} X  $ and  $S\curvearrowright_{\rho} Y $  are conjugate, then they are continuously orbit equivalent.
\end{proposition}
\begin{proof}
	Assume that  $   P\curvearrowright_{\theta} X  $ and  $S\curvearrowright_{\rho} Y $  are conjugate by maps $\varphi$ and $\alpha$. For $g\in G$, there exist $a,b\in P$ such that $g=ab^{-1}$.   One can check that  $\beta:g\in G\rightarrow \alpha(a)\alpha(b)^{-1}\in H$ is a well-defined  group isomorphism. Define $a(x,g)=\beta(g),b(y,h)=\beta^{-1}(h)$ for $(x,g)\in X\rtimes P$, $(y,h)\in Y\rtimes S$.

	For each $(x,g)\in X\rtimes P$, there exist $m,n\in P$ such that $g=mn^{-1}$ and $\theta_{m}(x)=\theta_{n}(u(x,g))$. Then $\rho_{\alpha(m)}(\varphi(x))=\varphi(\theta_{m}(x))=\varphi(\theta_{n}(u(x,g)))=\rho_{\alpha(n)}(\varphi(u(x,g)))$, which implies $\varphi(u(x,g))=u(\varphi(x),a(x,g))$. In a similar way, we can show that $\varphi^{-1}(u(y,h))=u(\varphi^{-1}(y),b(y,h))$  for $(y,h)\in  Y\rtimes  S$. Thus $   P\curvearrowright_{\theta} X  $ and  $S\curvearrowright_{\rho} Y $  are  continuously orbit equivalent.
\end{proof}

Following \cite{article.8}, we define the topological freeness as follows:
\begin{definition}
	An injective action  $  P\curvearrowright_{\theta} X  $ is called topologically free if $\{x\in X: \mbox{ $x\neq u(x,g)$ for all $g\in Q_x$ with $g\neq e$}\}$ is dense in $X$.
\end{definition}
One can check that an injective action  $  P\curvearrowright_{\theta} X  $ is topologically free if and only if the groupoid $ X\rtimes P $ is topologically principal.

\begin{lemma}\label{L1}
	In Definition \ref{De1}, if injective actions   $   P\curvearrowright_{\theta} X  $ and  $S\curvearrowright_{\rho} Y $  are topologically free, then
	
	\begin{enumerate}
		
		\item[(i)] mappings $a$ and $b$ are continuous cocycles;
		\item[(ii)]  $b(\varphi(x),a(x,g))=g$, $a(\varphi^{-1}(y),b(y,h))=h$ for $(x,g)\in X\rtimes P $ and $(y,h)\in Y\rtimes S $.
	\end{enumerate}	
\end{lemma}

\begin{proof}
	(i) For  $(x_{0},g_{1}), (u(x_{0},g_{1}),g_{2})$ in $ X\rtimes P$,   $(x_{0},g_{1}g_{2})\in X\rtimes P$. Choose $s_{i},t_{i}\in S, h_{i}\in H$, $i=1,2,3$, such that $h_{1}=a(x_{0},g_{1})=s_{1}t_{1}^{-1}$, $h_{2}=a(u(x_{0},g_{1}),g_{2})=s_{2}t_{2}^{-1}$ and  $h_{3}=a(x_{0},g_{1}g_{2})=s_{3}t_{3}^{-1}$. From the continuity of $a$  and $u$, there exists an open neighbourhood $U$ of $x_{0}$ such that   $ (u(x,g_{1}),g_{2})\in X\rtimes P$ when $(x,g_{1})\in X\rtimes P$, $a(x,g_{1})=a(x_{0},g_{1})$  and $a(u(x,g_{1}),g_{2})=a(u(x_{0},g_{1}),g_{2})$   for  $x\in U$.
	Then $\varphi(u(x,g_{1}))=u(\varphi(x),h_{1})$,   $\varphi(u(u(x,g_{1}),g_{2}))=u(\varphi(u(x,g_{1})),h_{2})$ and $\varphi(u(x,g_{1}g_{2}))=u(\varphi(x),h_{3})$, these follow that $\rho_{s_{1}}(\varphi(x))=\rho_{t_{1}}(\varphi(u(x,g_{1})))$, $\rho_{s_{2}}(\varphi(u(x,g_{1})))=\rho_{t_{2}}(\varphi(u(u(x,g_{1}),g_{2})))$
	and $\rho_{s_{3}}(\varphi(x))=\rho_{t_{3}}(\varphi(u(x,g_{1}g_{2})))$
	for each $x\in U$.
	Let $t_{1}^{-1}s_{2}=mn^{-1}$, $m,n\in S$. Then $ h_{1}h_{2}= (s_{1}m)(t_{2}n)^{-1}$  and $\rho_{s_{1}m}(\varphi(x))%=\rho_{t_{1}m}(\varphi(u(x,g_{1})))=\rho_{s_{2}n}(\varphi(u(x,g_{1})))
	=\rho_{t_{2}n}(\varphi(u(u(x,g_{1})),g_{2}))$. Since    $\rho_{s_{1}m}(\varphi(x))=\rho_{t_{2}n}(u(\varphi(x)),h_{1}h_{2})$, we see that $u(\varphi(x),h_{1}h_{2})=u(\varphi(x),h_{3})$  for each $x\in U$. Topological freeness of $S\curvearrowright_{\rho}Y $ implies $h_{3}=h_{1}h_{2}$, i.e., $a(x_{0},g_{1}g_{2})=a(x_{0},g_{1})a(u(x_{0},g_{1}),g_{2})$. In the same way as above, we can show that $ b $ is a cocycle.
	
	(ii)  From equations (\ref{e1}) and (\ref{e2}), one   sees that $u(x,g)= u(x,b(\varphi(x),a(x,g)))$ for $(x,g)\in X\rtimes P$. By the continuity of $ a $ and $ b $,  this equation holds for some open neighbourhood $U$ of $x$. Topological freeness of $ P\curvearrowright_{\theta}X $ implies  $b(\varphi(x),a(x,g))=g$.  In the same way, we can show that  $a(\varphi^{-1}(y),b(y,h))=h$ for $(y,h)\in Y\rtimes S $.

\end{proof}

\begin{theorem}
	Let $P\curvearrowright_{\theta} X  $ and  $S\curvearrowright_{\rho} Y $   be two topologically free injective actions. Then the followings are equivalent:
	\begin{enumerate}
		\item[(i)]   $   P\curvearrowright_{\theta} X  $ and  $S\curvearrowright_{\rho} Y $  are continuously orbit equivalent;
		
		\item[(ii)]	$X\rtimes  P $ and $Y\rtimes S$ are isomorphic as \'{e}tale groupoids;
		\item[(iii)]  There is  a $C^{*}$-isomorphism $\Phi :\; C^{*}_{r}(X\rtimes P) \rightarrow  C^{*}_{r}(Y\rtimes S ) $ such that $\Phi(C(X))=C(Y)$;
		\item[(iv)] 	There is  a $C^{*}$-isomorphism $\Psi:\; C(X)\rtimes^{\widehat{\alpha}}_{r}G \rightarrow  C(Y)\rtimes^{\widehat{\beta}}_{r} H  $ such that $\Psi(C(X))=C(Y)$, where $\widehat{\beta}$ is the partial action of $H$ on $C(Y)$.
	\end{enumerate}	
	Moreover, (ii) $ \Rightarrow$ (i) holds  without the assumption of topological
	freeness.
\end{theorem}

\begin{proof}
	(i) $ \Rightarrow $ (ii) Let $\varphi, a $ and $ b $ be three maps implementing the continuous  orbit equivalence of  $   P\curvearrowright_{\theta} X  $ and  $S\curvearrowright_{\rho} Y $.  By Lemma \ref{L1}, maps $X\rtimes P\rightarrow Y\rtimes S, (x,g)\mapsto (\varphi(x),a(x,g))  $ and $Y\rtimes S\rightarrow X\rtimes P, (y,h)\mapsto (\varphi^{-1}(y),b(y,h))$ are continuous groupoid homomorphisms, and they    are inverse to each other. Hence $X\rtimes  P $ and $Y\rtimes S$ are  isomorphic as \'{e}tale groupoids.
	
	(ii) $ \Rightarrow $ (i)  Assume that  $ \Lambda :X\rtimes P\rightarrow Y\rtimes  S   $ is an  isomorphism. Let $ \varphi $ be the restriction of $\Lambda$ to the unit space $X$, and let $a=c\Lambda$, $b=c\Lambda^{-1}$.  Then $\varphi: X\rightarrow Y$ is a homeomorphism, and $a: X\rtimes P\rightarrow  H$, $b: Y\rtimes  S\rightarrow G $ are continuous cocycles. Moreover, $\Lambda(x,g)=(\varphi(x),a(x,g))$, $\Lambda^{-1}(y,h)=(\varphi^{-1}(y),b(y,h))$. Then $$\varphi(u(x,g))=\Lambda(s(x,g))=s(\Lambda(x,g))= u(\varphi(x),a(x,g)),$$
	$$\varphi^{-1}(u(y,h))=\Lambda^{-1}(s(y,h))=s(\Lambda^{-1}(y,h))=u(\varphi^{-1}(y),b(y,h)).$$ Thus  $   P\curvearrowright_{\theta} X  $ and  $S\curvearrowright_{\rho} Y $  are continuously orbit equivalent. This does not use topological freeness.

	The equivalence of (ii), (iii) and (iv)   follows from \cite{article.13} and Theorem 2.10.	
	
\end{proof}

The following example comes from \cite{article.14}.
\begin{example}
	
	Denote by $ \mathbb{Q}_{+}=\{x\in  \mathbb{Q}:x\geq 0\} $, 	  $ \mathbb{Q}_{+}^{*}=\{x\in  \mathbb{Q}:x > 0\} $, $ \mathbb{Q}_{\geq1}=\{x\in  \mathbb{Q}:x\geq 1\} $, and $ \mathbb{N}^{*}=\mathbb{N}\setminus \{0\}$.  Let
	$$G=\left\lbrace  \left[  \begin{array}{cc}
		a	& b \\
		0	& 1
	\end{array}\right] : a\in \mathbb{Q}_{+}^{*},b\in \mathbb{Q}
	\right\rbrace  $$
	be the semi-direct of the additive group $ \mathbb{Q}$ by the multiplication action of  $\mathbb{Q}_{+}^{*}  $. Let
	$$P_{1}=\left\lbrace  \left[  \begin{array}{cc}
		a	& b \\
		0	& 1
	\end{array}\right] : a\in \mathbb{Q}_{\geq1},b\in \mathbb{Q}_{+}
	\right\rbrace ,$$ $$P_{2}=\left\lbrace  \left[  \begin{array}{cc}
		a	& b \\
		0	& 1
	\end{array}\right] : a\in \mathbb{N}^{*},b\in \mathbb{Q}_{+}
	\right\rbrace . $$
	Then $  P_{1}, P_{2} $ are unital semi-groups of $G$ and $G=P_{i}P_{i}^{-1}$, $ i=1,2 $.

	Let $X=\left[ -\infty,0\right] \times \left[ 0,1\right]  $, where $ \left[ -\infty,0\right]  $ is the one-point compactification of $ \left(  -\infty,0\right]  $. For $ \left[  \begin{array}{cc}
		a	& b \\
		0	& 1
	\end{array}\right] \in G  $ and $ (x,y)\in X $,   define $$(x,y)\ast_{\theta}
	\left[  \begin{array}{cc}
		a	& b \\
		0	& 1
	\end{array}\right]
	=(\dfrac{x-b}{a},\dfrac{y}{a}). $$
	Then $ \theta $ is not an action of $G$ on $X$, but $	P_{1} \curvearrowright_{\theta}X  $ and $	P_{2} \curvearrowright_{\theta}X  $ are right injective actions.
	Note that both for $	P_{1} \curvearrowright_{\theta}X  $ and $	P_{2} \curvearrowright_{\theta}X  $,  $ Q_{(x,y)}=\left\lbrace  \left[  \begin{array}{cc}
		a	& b \\
		0	& 1
	\end{array}\right]\in G :   a\geq y,b\geq x
	\right\rbrace $ for  $ (x,y)\in X$, and  for all $ g\in Q_{(x,y)}$ with $g$ is not the identity matrix, $ (x,y)\neq u((x,y), g)$. Then $	P_{1} \curvearrowright_{\theta}X  $ and $	P_{2} \curvearrowright_{\theta}X  $ are all topologically free.
	
	Let  $   \left[  \begin{array}{cc}
		a	& b \\
		0	& 1
	\end{array}\right] $ be in $P_{1}$ or  $P_{2} $  arbitrary,  observe that $X\ast_{\theta}
	\left[  \begin{array}{cc}
		a	& b \\
		0	& 1
	\end{array}\right]
	=\dfrac{1}{a}( [-\infty,-b]\times[0,1]) $ is open in $X$. Then it follows   that    transformation groupoids $  X\rtimes P_{1} $ and $  X\rtimes P_{2} $ are  all   \'{e}tale.
	
\end{example}

\begin{proposition}
	$	P_{1} \curvearrowright_{\theta}X  $ and $	P_{2} \curvearrowright_{\theta}X  $ are continuously orbit equivalent, but they are not conjugate.	Moreover, $  X\rtimes P_{1} $ and $  X\rtimes P_{2} $ are isomorphic as \'{e}tale groupoids.
\end{proposition}
\begin{proof}
	
	Since $G=P_{1}P_{1}^{-1}=P_{2}P_{2}^{-1}$, one can check that $ (x,g)\in X\rtimes P_{1}$ if and only if  $(x,g)\in X\rtimes P_{2}$. It is then easy to see that  $  X\rtimes P_{1} $ and   $ X\rtimes P_{2}$  are \'{e}tale groupoid isomorphic and thus 	$	P_{1} \curvearrowright_{\theta}X  $ and $	P_{2} \curvearrowright_{\theta}X  $ are continuously orbit equivalent.
	
	An easy check shows that	 if  injective actions $  P\curvearrowright_{\theta} X  $ and  $S\curvearrowright_{\rho} Y $  are conjugate by homeomorphism $\varphi$ and semi-group isomorphism $\alpha$, assume that there exists $x_{0}\in X$ such that $Q_{x_{0}}=P$. Then $Q_{\varphi(x_{0})}=S$.
	In this case, for $(0,1)\in  X $, note that $Q_{(0,1)}=P_{1}\neq P_{2}$, then   $	P_{1} \curvearrowright_{\theta}X  $ and $	P_{2} \curvearrowright_{\theta}X  $ are   not conjugate.
\end{proof}

In the rest of this section, we discuss the  classification of the injective actions and the associated group actions  in Section 2 up to conjugacy,   no further results have been obtained for their  continuous orbit equivalence.

\begin{proposition}
	If injective actions $ P\curvearrowright_{\theta} X  $ and $S\curvearrowright_{\rho} Y  $ are conjugate, then   $ G\curvearrowright_{\tilde{\theta}} \widetilde{X}  $ and $H\curvearrowright_{\tilde{\rho}} \widetilde{Y}  $ are conjugate.
\end{proposition}
\begin{proof}
	Let $\varphi: X\rightarrow Y$ be a  homeomorphism  and $\alpha: P\rightarrow S$  be a semi-group isomorphism such that  $\varphi(\theta_{m}(x))=\rho_{\alpha(m)}(\varphi(x))$ for    $x\in X$ and $m\in P$.
	For $g\in G$, there exist $a,b\in P$ such that $g=ab^{-1}$. Define $\beta:g\in G\rightarrow \alpha(a)\alpha(b)^{-1}\in H$. One can check that $\beta$ is a well-defined group isomorphism and $(x,g)\sim(y,h)$ in $ \widetilde{X} $ if and only if $ (\varphi(x),\beta(g))\sim( \varphi(y),\beta(h))$ in $  \widetilde{Y}$.  We can therefore define a map
	$$\widetilde{\varphi}:  \widetilde{X}\rightarrow  \widetilde{Y}, \; [x,g] \mapsto [\varphi(x),\beta(g)]$$
	and $\widetilde{\varphi}$ is  bijective  with inverse $ \widetilde{\varphi}^{-1} $, defined by $ \widetilde{\varphi}^{-1}([y,h])=[\varphi^{-1}(y),\beta^{-1}(h)] $.

	To see that $\widetilde{\varphi}$ is continuous,  it suffices to show that   $\widetilde{\varphi}\circ\pi$ is continuous, where $ \pi :X\times G \rightarrow \widetilde{X}$ is the quotient map. Suppose $(x_{n},g_{n})\rightarrow(x,g)$ in $X\times G$. Then   $g_{n}= g$ for large $n$, so we can assume that $g_{n}=g$. Hence $(\varphi(x_{n}),\beta(g))\rightarrow (\varphi(x),\beta(g)) $ in $Y\times H$. Since map $(y,h)\in Y\times H\rightarrow [y,h]\in \widetilde{Y}$ is continuous, we have $[\varphi(x_{n}),\beta(g )] \rightarrow [\varphi(x),\beta(g)] $. Thus $\widetilde{\varphi}$ is continuous. In a similar way, we can show that  $\widetilde{\varphi}^{-1}$ is continuous and thus $\widetilde{\varphi}$ is a homeomorphism.   Finally, for $[x,g]\in \widetilde{X}$ and $h\in G$,
	$$\widetilde{\varphi}(\tilde{\theta}_{h}[x,g])=\widetilde{\varphi}([x,gh])=[\varphi(x),\beta(gh)]= \tilde{\rho}_{\beta(h )}([\varphi(x),\beta(g )])=\tilde{\rho}_{\beta(h )}(\widetilde{\varphi}[x,g]).$$ Hence $ G\curvearrowright_{\tilde{\theta}} \widetilde{X}  $ and $H\curvearrowright_{\tilde{\rho}} \widetilde{Y}  $ are conjugate.
	
\end{proof}

\begin{proposition}
	Let  $P\curvearrowright_{\theta}X$ be a right   action by homeomorphisms as in Remark \ref{R3}, then $G\curvearrowright_{\hat{\theta}} X$  and  $ G\curvearrowright_{ \tilde{\theta}} \widetilde{X}$ are conjugate.
\end{proposition}
\begin{proof}
	Define  $ \phi :(x,g) \in X\times G\rightarrow \hat{\theta}_{g}(x)\in X$ and $ \varphi:x\in X\rightarrow  [x,e]\in \widetilde{X}$. Then  $ \phi$ is continuous and $ \varphi $ is bijective. Moreover, $\varphi(\phi(x,g))=\varphi( \hat{\theta}_{g}(x))=[\hat{\theta}_{g}(x),e]=[x,g]=\pi(x,g)$ for $(x,g)\in X\times G$. Let $U$ be an open subset in  $\widetilde{X}$. Since $ \pi^{-1}(U)=  \phi^{-1}(\varphi^{-1}(U) )$ is open in $ X\times G $, the continuity of $\phi$ implies that $ \varphi^{-1}(U) $ is  open in $ X $, it follows that $\varphi   $ is continuous. By compactness of  $X$,  we conclude that   $ \varphi $ is a homeomorphism. Furthermore, for $x\in X,g\in G$, we have
	$$\varphi(\hat{\theta}_{g}(x))=[\hat{\theta}_{g}(x),e]=[x,g]=\widetilde{\theta}_{g}([x,e])=\widetilde{\theta}_{g}(\varphi(x)).$$ Therefore  $ G\curvearrowright_{\hat{\theta}} X$  and  $ G\curvearrowright_{ \tilde{\theta}} \widetilde{X}$ are conjugate.
	
\end{proof}

\section{Injective actions on compactifications  of semi-groups}
In Section 2, we see that each injective right action of a semi-group can be  dilated to be  a group  action. Recall that a compactification of a   locally compact Hausdorff space $Z$ is a  compact Hausdorff space   containing a dense continuous image of $Z$. In this section, we consider   the right injective  action  of  a semi-group on its compactifications.

Let $G$ be a countable group, $P$ be a right Ore sub-semigroup of $G$ and $G=PP^{-1}$. Denote by  $l^{\infty}(G)$ the set of all  bounded  complex valued functions on $G$. It is  a unital abelian $C^{*}$-algebra. Let  $\rho_{g}$ be the operator on $l^{\infty}(G)$ such that $\rho_{g}(\xi)(h)=\xi(hg)$ for $\xi\in l^{\infty}(G)$ and $ g,h\in G$. Then $G\curvearrowright_{\rho}l^{\infty}(G)$ is a   group action by $*$-isomorphisms.

For any unial $C^{*}$-subalgebra $\mathcal{A}$ of $ l^{\infty}(G) $, let $\Sigma_{\mathcal{A}} $ be the maximal ideal space of $\mathcal{A}$, or, equivalently the pure (or, multiplicative) state space of $\mathcal{A}$. Then $\Sigma_{\mathcal{A}}$ is compact Hausdorff and $\mathcal{A}$ is isomorphic to $C(\Sigma_{\mathcal{A}})$ by Gelfand-Naimark theory. Moreover, we have a map $g\in G\rightarrow \hat{g}\in\Sigma_{\mathcal{A}} $, where  $ \hat{g} (\xi)=\xi(g)$ for $\xi\in \mathcal{A} $, whose range is dense in $ \Sigma_{\mathcal{A}} $, i.e., $\Sigma_{\mathcal{A}} $ is  a compactification of $G$ (with discrete topology). Assume that $\mathcal{A}$ is invariant under $\rho$, i.e. $ \rho_{g}( \mathcal{A})=\mathcal{A} $ for each $g\in G$. Then the automorphism action of $G$ on $\mathcal{A}$ induces an action $\theta$ of $G$ on $\Sigma_{\mathcal{A}}$ by homeomorphisms, defined by:
$$\theta_{g}(\hat{h})=\widehat{hg} \mbox{ for $g\in G$ and $\hat{h}\in \Sigma_{\mathcal{A}}$}.$$
Let $X$ be the closure of $\{\hat{a}:a\in P\}$ in $ \Sigma_{\mathcal{A}} $. Then $X$ is a compactification of $P$ (with discrete topology). Moreover, $\theta_{a}(X)\subseteq X$ and the map $\theta_{a}:X\rightarrow X$ is injective for each $a\in P$, and $\theta_{a}\theta_{b}=\theta_{ba}$ for all $a,b\in P$, i.e., $ P\curvearrowright_{\theta} X$ is a right injective action of $P$.

We know that if the above $C^*$-algebra $\mathcal{A}$ is countably generated then $ \Sigma_{\mathcal{A}} $ is  second-countable and metrizable.
For $S\subset G$, let $f=\chi_S\in l^{\infty}(G)$ be the characteristic function on $S$ and $\mathcal{A}_{f}$ be the unital  $C^*$-algebra generated by $\{I,\rho_g(f):\, g\in G\}$ in $l^{\infty}(G)$, where $I$ is the unit of $ l^{\infty}(G) $. Then $\mathcal{A}_{f}  $ is  invariant under $\rho$. Let $ \Sigma_{\mathcal{A}_{f}}$, $X$, and $ P\curvearrowright_{\theta} X$ be as in the above paragraph. Let $Y$ be the closure of $\{\widehat{g}:\, g\in S\}$ in $\Sigma_{\mathcal{A}_{f}}$.  For $\gamma\in \Sigma_{\mathcal{A}_{f}}$, we let $$A_{\gamma}=\{h\in G: \gamma(\rho_{h^{-1}}(f))=1\}=\{h\in G:\, \gamma(\chi_{Sh})=1\}.$$ In particular, $A_{\widehat{g}}=S^{-1}g$ for $g\in G$. We consider the shift action $\beta$ of $G$ on $\{0,1\}^G$ by
$$\beta_{g}(\xi)(h)=\xi(hg^{-1}),\, \mbox{for $g, h\in G$, $\xi\in \{0,1\}^G$}.$$
One can check  the map $\pi:\, \gamma\in  \Sigma_{\mathcal{A}_{f}}\rightarrow \chi_{A_{\gamma}}\in \{0,1\}^{G}$ is continuous, injective and $G$-equivariant, i.e., $\beta_g\pi=\pi\theta_g$ for each $g\in G$. Put
$$\widetilde{ \Sigma}=\pi(\Sigma_{\mathcal{A}_{f}}),\,\, \widetilde{X}=\pi(X), \widetilde{Y}=\pi(Y).$$
Remark that $\widetilde{Y}$ is the closure of $\{\chi_{S^{-1}h}:\, h\in S\}$ in $\{0,1\}^G$.

\begin{theorem} Assume that $P\subseteq S\subseteq G$ and $SP\subseteq S$. Then $P\curvearrowright_{\theta} Y$ is an injective right action of $P$ on $Y$ whose transformation groupoid $Y\rtimes P$ is \'{e}tale. In particular, when
	$S=P$, $P\curvearrowright_{\theta} X$ is conjugate to the right action of $P$ on the order compactification of $P$.
\end{theorem}

\begin{proof} Since $\theta_a(\widehat{h})=\widehat{ha}\in Y$ for $a\in P$ and $h\in S$, it follows from the assumption that $\theta_a(Y)\subseteq Y$. Thus $P\curvearrowright_{\theta} Y$ is an injective right action, and $P\curvearrowright_{\theta} Y$ and $P\curvearrowright_{\beta} \widetilde{Y}$ are conjugate. Next we show that $\theta_a(Y)$ is open in $Y$ for each $a\in P$.
	
	We claim that $\beta_a(\widetilde{Y})=\{\xi\in \widetilde{Y}:\, \xi(a)=1\}$ for $a\in P$.
	
	In fact, fix $a\in P$, for $\xi\in \beta_a(\widetilde{Y})\subseteq \widetilde{Y}$, choose $\eta\in \widetilde{Y}$ with $\xi=\beta_a(\eta)$. Since $e\in S^{-1}h$ for each $h\in S$, we have $\chi_{S^{-1}h}(e)=1$ for each $h\in S$, thus $\eta(e)=1$. Hence
	$\xi(a)=\beta_a(\eta)(a)=\eta(e)=1$.
	
	On the other hand, for $\xi\in \widetilde{Y}$ with $ \xi(a)=1$, we choose $\{a_n\}\subset S$ with $\chi_{S^{-1}a_n}\rightarrow \xi$ in $\{0,1\}^G$. So $\chi_{S^{-1}a_n}(a)\rightarrow \xi(a)=1$. It follows   that there exists $N$ such that $a\in S^{-1}a_n$ for every $n\geq N$. Choose $b_n\in S$ such that $a=b_n^{-1}a_n$ for $n\geq N$. Also since $\widetilde{Y}$ is compact, we can assume that $\{\chi_{S^{-1}b_n}\}$ converges to $\varsigma$ in $\widetilde{Y}$. It follows from the continuity of the action $\beta$ that $\beta_a(\chi_{S^{-1}b_n})\rightarrow \beta_a(\varsigma)$, which implies that $\chi_{S^{-1}b_na}\rightarrow \beta_a(\varsigma)$. Thus $\xi=\beta_a(\varsigma)$. We finish the claim.
	
	From the claim, we have $\beta_a(\widetilde{Y})=\widetilde{Y}\cap\{\xi\in \{0,1\}^G:\, \xi(a)=1\}$, thus $\beta_a(\widetilde{Y})$ is open in $\widetilde{Y}$. Consequently,  $\theta_a(Y)$ is open in $Y$ for each $a\in P$,  and transformation groupoid $Y\rtimes P$ is \'{e}tale. From the above proof and the argument to  the order compactification of $P$ in \cite{article.14}, when $S=P$, $P\curvearrowright_{\theta} X$ is conjugate to the right action of $P$ on the order compactification of $P$.
\end{proof}

Let $ P_{\infty}=P\cup \{\infty\}$ be the one-point compactification of $P$.  We consider the right action $ P \curvearrowright_{\sigma}P_{\infty}$, defined by  $$\sigma_{a}(b)=ba,\;\sigma_{a}(\infty)=\infty,\;\;\text{for}\; a,b\in P.$$
Remark that $\sigma_{a}( P_{\infty})=Pa\cup \{\infty\}$ is open in $  P_{\infty} $ when $P \setminus Pa$ is finite for each $a\in P$, in this case the transformation gropoid associated to the injective action is \'{e}tale. For a complex-valued function $\xi$ on $G$, we say that $\lim_{g\rightarrow \infty}\xi(g)$ exists if there exists a complex number $\lambda$ such that, for any $\varepsilon>0$, there exists a finite subset $E$ of $G$ such that $|\xi(g)-\lambda|<\varepsilon$ for all $g\notin E$. In this case, we write $ \lim_{g\rightarrow \infty}\xi(g)=\lambda $.

\begin{proposition} Assume that $S=\{e\}$ and $f=\delta_e\in l^{\infty}(G)$ is the characteristic function on $\{e\}$. Then $\mathcal{A}_{f} =\{\xi\in  l^{\infty}(G): \lim_{g\rightarrow \infty}\xi(g) \;\text{exists}\} $, and $\Sigma_{\mathcal{A}_f}$ and $X$ are homeomorphic to the one-point compactifications of $G$ and $P$, respectively. Moreover, the injective action $P\curvearrowright_{\theta} X$ is conjugate to $ P \curvearrowright_{\sigma}P_{\infty}$. Thus the transformation groupoid $X\rtimes P$ is \'{e}tale if and only if $P \setminus Pa$ is finite for each $a\in P$.
\end{proposition}

\begin{proof} Remark that $\rho_{g}(f)=\delta_{g^{-1}}$ for each $g\in G$. Thus $\mathcal{A}_f$ is the closure of $\mathbb{C}G$ under the supremum norm, so $\mathcal{A}_{f} =\{\xi\in  l^{\infty}(G): \lim_{g\rightarrow \infty}\xi(g) \;\text{exists}\} $. One can check   the map $g\in G\rightarrow \widehat{g}\in \Sigma_{\mathcal{A}_f}$ is injective. Define $\gamma_0(\xi)=\lim_{g\rightarrow \infty}\xi(g)$ for $\xi\in \mathcal{A}_f$. We have that $\gamma_{0}\in \Sigma_{\mathcal{A}_f}$, $\Sigma_{\mathcal{A}_f}=\{\widehat{g}:\,g\in G\}\cup\{\gamma_0\}$. Thus $\Sigma_{\mathcal{A}_f}$ and $X$ are homeomorphic to the one-point compactifications of $G$ and $P$, respectively.  By the action of $\theta$, we have $P\curvearrowright_{\theta} X$ is conjugate to $ P \curvearrowright_{\sigma}P_{\infty}$.
	
\end{proof}

From \cite[Proposition 5.1]{article.14}, the injective action of a semi-group on its ordered compactification is determined uniquely up to conjugacy by three conditions. For infinite countable right Ore sub-semigroups $P$ and $S$, respectively, of two groups $G$ and $H$, one can check that, two actions $ P \curvearrowright_{\sigma}P_{\infty}$   and $S \curvearrowright_{\rho}S_{\infty}$ are orbit equivalent, and   they are conjugate if and only if $P$ and $S$ are semi-group isomorphic. Moreover, if two actions are continuously orbit equivalent then $G$ and $H$ are isomorphic.

\begin{theorem}
	Assume that $P \setminus Pa$ is finite for each $a\in P$. Let $ P\curvearrowright_{\rho} X$ be an injective right action of $P$ on a compact Hausdorff space $X$.  If
	\begin{enumerate}
		\item[(i)] there exists a unique $x_{\infty}$ in $X$ such that $Q_{x_{\infty}}=G$, and
		\item[(ii)] there exists   $x_{0}$ in $X$ such that the map $a\in P\rightarrow \rho_{a}(x_{0})\in X$ is injective and has dense range in $X$,
	\end{enumerate}
	then $ P \curvearrowright_{\sigma}P_{\infty}$   and $ P\curvearrowright_{\rho} X$  are conjugate.
\end{theorem}
\begin{proof}
	Define $ \Lambda: P_{\infty}\rightarrow X $ by $\Lambda(a)=\rho_{a}(x_{0})$ for $a\in P$, and $\Lambda(\infty) =x_{\infty}$. Then $\Lambda$ is continuous on $P$. It suffices therefore to show that $\Lambda$ is continuous  at $ \infty $. For an arbitrary open neighbourhood $U$ of $x_{\infty}$ in $X$, we only need to show the set $F=\{a\in P:\rho_{a}(x_{0})\notin U\}$ is finite. For otherwise, if $F$ is infinite, we can choose a sequence $\{a_{n}\}$ in $F$, where $x_{n}\neq x_{m}$ for $n\neq m$, such that $\{a_{n}\}$ converges to $\infty$ in $P$. By \cite[Remark 4.5]{article.14}, the map $z\in P_{\infty}\rightarrow Q_z\in \mathcal{C}(G)$ is continuous where $\mathcal{C}(G)$ is the space of all subsets of $G$ with the Vietoris topology. Thus  $Q_{a_{n}}=a_{n}^{-1}P\rightarrow Q_{\infty}=G$. Note that $X$ is compact, so we can assume that $ \rho_{a_{n}}(x_{0})\rightarrow x$ in $X$, it follows from the continuity of the map $x\in X\rightarrow Q_x\in \mathcal{C}(G)$ that $ Q_{\rho_{a_{n}}(x_{0})} \rightarrow Q_{x}$. Since $ \rho_{a_{n}}(x_{0})\notin U $, we have $x\neq x_{\infty}$.  Also since $ Q_{\rho_{a_{n}}(x_{0})} =a_{n}^{-1}Q_{x_{0}}=a_{n}^{-1}P\rightarrow G$, we see that $Q_{x}=G$, in contradiction with condition (ii). Thus $F$ is finite. Hence $\Lambda$ is continuous and thus is a homeomorphism.
	
	Moreover, for each $a,m\in P$,
	$$\rho_{a}\Lambda(m)=\rho_{a}\rho_{m}(x_{0})=\rho_{ma}(x_{0})=\Lambda(ma)=\Lambda(\sigma_{a}(m)).$$ We claim that $\rho_{a}(x_{\infty})=x_{\infty}$ for each $a\in P$. In fact, write $\rho_{a}(x_{\infty})=u$. Then $Q_{u}=Q_{\rho_{a}(x_{\infty})}=a^{-1}Q_{x_{\infty}}=G$, it follows that $u=x_{\infty}$. Hence $\rho_{a}\Lambda(\infty)=\rho_{a}(x_{\infty})=x_{\infty}=\Lambda(\sigma_{a}(\infty))$ for each $a\in P$, so we conclude that $ P \curvearrowright_{\sigma}P_{\infty}$   and $ P\curvearrowright_{\rho} X$  are conjugate.
	
\end{proof}

\begin{example}\label{E1}
	Let $\mathbb{Z}$ and $\mathbb{N}$ denote  the (additive) group of integer numbers and its sub-semigroup of natural numbers respectively. Let $\mathbb{Q}^*_{>0}$ and $\mathbb{N}^*_{>0}$ denote  the (multiplication) group of positive rational numbers  and its sub-semigroup of positive integer numbers respectively. Let $\mathbb{N}^*_{\infty}$ and $\mathbb{N}_{\infty}$ denote the  one-point compactification of $\mathbb{N}^{*}_{>0}$ and $\mathbb{N}$ respectively.
	
	Consider the right injective actions $  \mathbb{N}\curvearrowright_{\theta}   \mathbb{N}^*_{\infty}$ and $ \mathbb{N}^{*}_{>0}\curvearrowright_{\rho}   \mathbb{N}^*_{\infty}$, defined by
	$$\theta_{m}(n)=n+m, \theta_{m}(\infty)=\infty,\;\; \text{for}\; m \in \mathbb{N} , n\in \mathbb{N}^{*}_{>0}.$$
	$$\rho_{m}(n)=nm, \rho_{m}(\infty)=\infty \;\; \text{for}\;\; m,n\in \mathbb{N}^{*}_{>0}.$$
	Then both of them are topologically free, and for $  \mathbb{N}\curvearrowright_{\theta}   \mathbb{N}^*_{\infty}$, $$  Q_{k}=
	\{-k+1,-k+2,-k+3,\cdots\}, \; [k]_{\theta}=\mathbb{N}^*_{>0}\;	\text{for} \;k\in \mathbb{N}^*_{>0},$$ $$ Q_{\infty}=\mathbb{Z},\; [\infty]_{\theta} = \{\infty\}.$$ And for $ \mathbb{N}^{*}_{>0}\curvearrowright_{\rho}   \mathbb{N}^*_{\infty}$ ,  $$  Q_{k}=
	\{1/k,2/k,3/k,\cdots\},\; [k]_{\rho} = \mathbb{N}^*_{>0}\; 	\text{for}\; k\in \mathbb{N}^*_{>0},$$    $$ Q_{\infty}=\mathbb{Q}^*_{>0},\; [\infty]_{\rho} = \{\infty\}.  $$
	
	Similarly, we can define right injective actions $\mathbb{N}\curvearrowright_{\theta}   \mathbb{N}_{\infty}$ and $ \mathbb{N}^{*}_{>0}\curvearrowright_{\rho}   \mathbb{N}_{\infty}$.
	Note that $\mathbb{N}^*  \setminus ( \mathbb{N}^* + m)$ and $\mathbb{N}   \setminus ( \mathbb{N}  + m)$ are finite for each $m\in \mathbb{N}$, but   $\mathbb{N}^*  \setminus  \mathbb{N}^*m$ and $\mathbb{N}   \setminus  \mathbb{N}m$ are not finite for each $m\in \mathbb{N}^*_{>0}$. Theoretically, we discuss the orbit  equivalence of injective actions under the conditions of Remark 2.3, but for this example we only  consider it in terms  of the definition of orbit equivalence, and we give the following result.

\end{example}

\begin{proposition}
	\begin{enumerate}
		\item[(i)] Injective actions $  \mathbb{N}\curvearrowright_{\theta}   \mathbb{N}^*_{\infty}$  and  $ \mathbb{N}^*_{>0}\curvearrowright_{\rho}   \mathbb{N}^*_{\infty}$  are   orbit equivalent, but they are not continuously orbit equivalent;
		
		\item[(ii)]Injective actions $  \mathbb{N}\curvearrowright_{\theta} \mathbb{N}_{\infty}$ and  $ \mathbb{N}^*_{>0}\curvearrowright_{\rho}  \mathbb{N}_{\infty}$ are not orbit equivalent.
	\end{enumerate}
\end{proposition}
\begin{proof}
	(i)	
	The 	identity mapping  $ \varphi$ on $\mathbb{N}^*_{\infty}$  implements  the orbit equivalence of    $\mathbb{N}\curvearrowright_{\theta}  \mathbb{N}^*_{\infty}$  and  $ \mathbb{N}^*_{>0}\curvearrowright_{\rho}  \mathbb{N}^*_{\infty}$.
	From  Example \ref{E1}, we see that $$\mathbb{N}^*_{\infty}\rtimes \mathbb{N}=\{(\infty,g),(k,-k+l)\;|\; g\in \mathbb{Z},k,l\in\mathbb{N}^*_{>0} \},$$
	$$\mathbb{N}^*_{\infty}\rtimes \mathbb{N}^*_{>0}=\{(\infty,g),(k, l/k)\;|\; g\in \mathbb{Q}^*_{>0} ,k,l\in\mathbb{N}^*_{>0} \}.$$
	Assume that $  \mathbb{N}\curvearrowright_{\theta}   \mathbb{N}^*_{\infty}$  and  $ \mathbb{N}^*_{>0}\curvearrowright_{\rho}   \mathbb{N}^*_{\infty}$  are  continuously orbit equivalent. Then there exist a homeomorphism $\varphi:\mathbb{N}^*_{\infty} \rightarrow \mathbb{N}^*_{\infty}$ and continuous maps $a:\mathbb{N}^*_{\infty}\rtimes  \mathbb{N}\rightarrow \mathbb{Q}^*_{>0}$, $b:\mathbb{N}^*_{\infty}\rtimes \mathbb{N}^*_{>0}\rightarrow \mathbb{Z}$ satisfy equations (\ref{e1}) and (\ref{e2}). In this case, $\varphi(\infty)=\infty$.

	For $k,l\in \mathbb{N}^*_{>0}$, since $ \varphi(l)=\varphi(u(k,-k+l))=u(\varphi(k),a(k,-k+l))= \varphi(k)a(k,-k+l)$, we see that $a(k,-k+l)=\varphi(l)/\varphi(k)$. For any $g\in G$, by the continuity of $a$  at $(\infty,g)$, there exists $m_{g}\in \mathbb{N}^*_{>0}$ with $m_{g}+g\geq1$, such that  $a(k,g)=\psi(g)$ for all $k\geq m_{g}$. Then $\psi(g)= \varphi(k+g)/\varphi(k) $. Calculation gives $ \psi(n)=\psi(1)^{n}$ for $n\in \mathbb{Z}$, it follows that $\varphi  $ cannot be a homeomorphism on $\mathbb{N}^*_{\infty}$. Hence   $  \mathbb{N}\curvearrowright_{\theta}  \mathbb{N}^*_{\infty}$  and  $ \mathbb{N}^*_{>0}\curvearrowright_{\rho}  \mathbb{N}^*_{\infty}$    are not continuously orbit equivalent.

	(ii) For $  \mathbb{N}\curvearrowright_{\theta}  \mathbb{N}_{\infty}$,  note that only the orbit at  $ \infty $ is a single point set,  but for $ \mathbb{N}^*_{>0}\curvearrowright_{\rho}  \mathbb{N}_{\infty}$, the orbits at $0$ and $\infty$ are all  sets of single point. Therefore $  \mathbb{N}\curvearrowright_{\theta} \mathbb{N}_{\infty}$ and  $\mathbb{N}^*_{>0}\curvearrowright_{\rho}  \mathbb{N}_{\infty}$ are not orbit equivalent.

\end{proof}

	\subsection*{Acknowledgements}

This work was supported by the NSF of China (Grant No. 11771379, 11971419, 11271224).

\end{document}